\documentclass[12pt]{article}
\usepackage{amssymb,amsmath}
\usepackage{amsthm}
\usepackage{cases}
\usepackage{amsfonts}
\usepackage{graphicx}
\usepackage{cite,color,xcolor}
\usepackage[left=2.6cm,right=2.6cm,top=2.4cm,bottom=2.4cm]{geometry}
\usepackage[colorlinks,citecolor=blue,urlcolor=blue]{hyperref}
\usepackage[utf8]{inputenc}

\newtheorem{theorem}{Theorem}[section]

\newtheorem{lemma}{Lemma}[section]

\newtheorem{remark}{Remark}[section]

\usepackage{txfonts}
\newcommand{\bal}{\begin{align}}
\newcommand{\bbal}{\begin{align*}}
\newcommand{\beq}{\begin{equation}}
\newcommand{\eeq}{\end{equation}}
\newcommand{\bca}{\begin{cases}}
\newcommand{\eca}{\end{cases}}

\newcommand{\pa}{\partial}
\newcommand{\fr}{\frac}
\newcommand{\na}{\nabla}
\newcommand{\De}{\Delta}

\newcommand{\cd}{\cdot}
\newcommand{\ep}{\varepsilon}
\newcommand{\dd}{\mathrm{d}}

\newcommand{\ri}{\rightarrow}

\newcommand{\R}{\mathbb{R}}

\newcommand{\les}{\lesssim}

\newcommand{\D}{\mathrm{div}\ }

\newcommand{\f}{\left}
\newcommand{\g}{\right}

\begin{document}
\bibliographystyle{plain}
\title{Anomalous Dissipation for the d-dimensional Navier–Stokes Equations}

\author{Jinlu Li$^{1}$, Yanghai Yu$^{2,}$\footnote{E-mail: lijinlu@gnnu.edu.cn; yuyanghai214@sina.com(Corresponding author); mathzwp2010@163.com} and Weipeng Zhu$^{3}$\\
\small $^1$ School of Mathematics and Computer Sciences, Gannan Normal University, Ganzhou 341000, China\\
\small $^2$ School of Mathematics and Statistics, Anhui Normal University, Wuhu 241002, China\\
\small $^3$ School of Mathematics and Big Data, Foshan University, Foshan, Guangdong 528000, China}

\date{\today}
\maketitle\noindent{\hrulefill}

{\bf Abstract:} The purpose of this paper is to study the vanishing viscosity limit for the d-dimensional Navier--Stokes equations in the whole space:
 \begin{equation*}
\begin{cases}
\pa_tu^\ep+u^\ep\cd\na u^\ep-\ep\Delta u^\ep+\nabla p^\ep=0,\\
\D u^\ep=0.
\end{cases}
\end{equation*}
We aim to presenting a simple rigorous examples of initial data which generates the corresponding solutions of the Navier--Stokes equations do  exhibit anomalous dissipation. Precisely speaking, we show that there are (classical) solutions for which the dissipation rate of the kinetic energy is bounded away from zero.



\vskip0mm\noindent{\hrulefill}

\section{Introduction}
One of the most important laws in the study of developed turbulence is Kolmogorov's 4/5-law, and in the derivation of this law, a significant assumption is the zeroth law (see Frisch \cite{u} for example). The zeroth law of turbulence states that, in
the limit of zero viscosity, the rate of kinetic energy dissipation for solutions
to the incompressible Navier-Stokes equations is non-vanishing. This is one of the central ansatz of Kolmogorov's 1941 theory \cite{K2}.
In this paper we consider the vanishing viscosity limit for the d-dimensional incompressible Navier–Stokes equations
\begin{equation}\label{1}\tag{NS}
\begin{cases}
\pa_tu^\ep+u^\ep\cd\na u^\ep-\ep\Delta u^\ep+\nabla p^\ep=f^\ep,&\quad (t,x)\in \R^+\times\R^d, \\
\D u^\ep=0,&\quad (t,x)\in \R^+\times\R^d,\\
u^\ep(t=0)=u^{\ep}_0,&\quad x\in \R^d,
\end{cases}
\end{equation}
where $\ep>0$ is the viscosity, $u^\ep: [0,\infty)\times {\mathbb R}^d\rightarrow {\mathbb R}^d$ is the velocity of the fluid, $p^\ep: [0,\infty)\times {\mathbb R}^d\rightarrow {\mathbb R}$ is the pressure and $f^\ep: [0,\infty)\times {\mathbb R}^d\rightarrow{\mathbb R}$ is an external body force. Assuming that the solution is sufficiently smooth and the initial data $u^{\ep}_0$ satisfy a uniform $L^2$ bound (in other words the total kinetic energy at the initial time is bounded).

Since $u^\ep$ is divergence free, taking the dot product of the equation with $u^\ep$ and integrating over $\R^d$,
it is immediate to see that  the energy balance:
\bal\label{0}
\fr12\frac{\dd}{\dd t}\|u^\ep(t)\|^2_{L^2(\R^d)}=-\ep\|\nabla u^\ep(t)\|^2_{L^2(\R^d)}+\int_{\R^d}f^\ep\cdot u^\ep\dd x.
\end{align}
The second term appearing on the right-hand side is the total work of the force $f^\ep$, while
the first term is the {\it energy dissipation rate} due to the viscosity of the fluid. Throughout the
paper, we assume that the external body force $f^\ep=0$. From \eqref{0} we see that the $L^2$ energy decay of solutions is governed by
\bbal
\fr12\|u^\ep(t)\|^2_{L^2(\R^d)}-\fr12\|u_0\|^2_{L^2(\R^d)}=-\ep\int_0^t\|\nabla u^\ep(\tau)\|^2_{L^2(\R^d)}\dd\tau.
\end{align*}
A fundamental postulate of Kolmogorov's 1941 theory of fully developed turbulence
\cite{K1,K2,K3}, called the zeroth law of turbulence, is that the anomalous dissipation of
the kinetic energy holds, namely the following inequality
\bbal\liminf_{\ep\ri 0^+}\ep\int_0^t\|\nabla u^\ep(\tau)\|^2_{L^2(\R^d)}\dd\tau>0
\end{align*}
is valid for finite time $t$.

As for passive scalars, experimental and numerical observations of hydrodynamic
turbulence suggest that, in the limit of vanishing viscosity, the rate of kinetic energy dissipation becomes nonzero \cite{KI,KR1,KR2,PKW}, i.e. there exists $\eta > 0$ independent of $\ep$ such that, in turbulent regimes, a family of Leray-Hopf solutions $\f\{u^\ep\g\}_{\ep>0}$ satisfies
$$\ep\int^T_0 \int_{\R^d}|\na u^{\ep}(t,x)|^2\dd x\dd t\geq\eta>0.$$

This phenomenon of anomalous dissipation is so fundamental to our modern understanding of turbulence that it is often termed the ``zeroth law". Laboratory experiments and numerical simulations of turbulence both confirm the above zeroth law \cite{BV,E,KI,J15}. However, this heuristic phenomenon is difficult to capture mathematically, and only a few rigorous results are known.
Recently, in \cite{DE}, Drivas et.al., investigated the anomalous scalar dissipation for fluid velocities with H\"{o}lder regularity
and gave sufficient conditions for anomalous dissipation in terms of the mixing rates of the advecting flow.  Bru\`{e} and De Lellis \cite{EL} established the anomalous dissipation for the forced Navier-Stokes equations. Jeong and Yoneda \cite{JY1,JY2} considered the anomalous dissipation for the 3D \eqref{1} with zero external force in the framework of $2+\fr12$-dimensional flows. We would like to mention that all the results mentioned above are obtained on the Torus case. To the best of our knowledge, there are no known anomalous dissipation for \eqref{1} in the whole space (even in $\R^2$) where it is rigorously proved. Keeping this classical physical phenomenon in mind, in this paper, we compare the Navier-Stokes flow and the heat flow with the same initial data, and deduce anomalous dissipation in the vanishing viscosity limit of \eqref{1} under some particular conditions on the initial data.
\begin{theorem}\label{th1}
Let $d\geq2$. There exist a sequence of $C^\infty$-smooth initial data $u^{\ep_n}_0$ on $\R^d$ with $\ep_n=2^{-2n}$ satisfying a uniform $L^2$ bound such that there are families of smooth solutions $u^{\ep_n}(t,x)\in C^{\infty}([0,1]\times \R^d)$ to the zero forced \eqref{1} with the property that
\bal\label{ad}
\ep_n\int^1_0 \f\|\na u^{\ep_n}(t,x)\g\|^2_{L^2(\R^d)}\dd t\gtrsim \|u^{\ep_n}_0\|_{L^2(\R^d)}.
\end{align}
In particular, we have
\bal\label{ad1}
\liminf_{n\ri \infty}\ep_n\int^1_0 \f\|\na u^{\ep_n}(t,x)\g\|^2_{L^2(\R^d)}\dd t\geq \eta,
\end{align}
for some absolute constant $\eta>0$.
\end{theorem}
Now let us say some words on the key idea of proof of Theorem \ref{th1}. We compare the Navier--Stokes equation and the heat equation with the same initial data, and show that the solution $U$ to the associated heat equation displays anomalous dissipation, i.e.,
\bal\label{y1}
\ep_n\int^1_0\|\na U_n(\tau)\|^2_{L^2(\R^d)}\dd \tau\approx \int^1_0\f\|e^{\ep_n \tau\De}u^{\ep_n}_0\g\|^2_{L^2(\R^d)}\dd \tau>0.
\end{align}
The main difficulty in the proof of Theorem \ref{th1} is to construct the initial data satisfying the following
\begin{itemize}
  \item the $L^2$-norm of initial data is bounded uniformly with respect to the viscosity;
  \item the difference of the Navier-Stokes flow and the heat flow is sufficiently small.
\end{itemize}
To achieve the above two points, taking advantage of the scaling in the whole space, we choose suitably the initial data whose frequency is supported in the annulus $\{\xi:|\xi|^2\thicksim \ep_n^{-1}\}$. This naturally gives that
\bal\label{y2}
\int^1_0\f\|e^{\ep_n \tau\De}u^{\ep_n}_0\g\|^2_{L^2(\R^d)}\dd \tau\thickapprox \f\|u^{\ep_n}_0\g\|^2_{L^2(\R^d)}.
\end{align}
Lastly, we would like to emphasize that we have to cancel the effect of the advection term. The key observation is that, base on our construction of initial data, the advection term $U\cdot \na U$ can be seen as $u^{\ep_n}_0\cdot \na u^{\ep_n}_0$, which is sufficiently small.
\begin{remark}\label{re1}
To the best of our knowledge, Theorem \ref{th1} is the first result on the anomalous dissipation for \eqref{1} in the whole space. Compared with the torus case in \cite{JY1}, we take a different construction on the initial data.
\end{remark}
\begin{remark}\label{re2}
We should mention that we  present a simple deterministic example of initial data, for which anomalous dissipation for \eqref{1}
exists and is a purely diffusive effect. In fact, from \eqref{y1} and \eqref{y2} we find  that viscosity alone is enough to obtain the anomalous dissipation for \eqref{1} in the whole space.
\end{remark}
{\bf Notation.}\; For $X$ a Banach space and $I\subset\R$, we denote by $\mathcal{C}(I;X)$ the set of continuous functions on $I$ with values in $X$. Sometimes we will denote $L^p(0,T;X)$ by $L_T^pX$. The symbol $\mathrm{A}\lesssim (\gtrsim)\mathrm{B}$ means that there is a uniform positive ``harmless" constant $\mathrm{C}$ independent of $\mathrm{A}$ and $\mathrm{B}$ such that $A\leq(\geq) \mathrm{C}B$, and we sometimes use the notation $\mathrm{A}\approx \mathrm{B}$ means that $\mathrm{A}\lesssim \mathrm{B}$ and $\mathrm{B}\lesssim \mathrm{A}$. Let us recall that for all $u\in \mathcal{S}'$, the Fourier transform $\mathcal{F}u$ is defined by
$$
\mathcal{F}u(\xi)=\int_{\R^d}e^{-\mathrm{i}x\cd \xi}u(x)\dd x \quad\text{for any}\; \xi\in\R^d.
$$Choose a radial, non-negative, smooth function $\vartheta:\R^d\mapsto [0,1]$ such that
${\rm{supp}} \;\vartheta\subset B(0, 4/3)$ and $\vartheta(\xi)\equiv1$ for $|\xi|\leq3/4$.
Setting $\varphi(\xi):=\vartheta(\xi/2)-\vartheta(\xi)$, then we deduce that
${\rm{supp}} \;\varphi\subset \left\{\xi\in \R^d: 3/4\leq|\xi|\leq8/3\right\}$ and
 $\varphi(\xi)\equiv 1$ for $4/3\leq |\xi|\leq 3/2$. For more details see \cite{BCD}.
The nonhomogeneous dyadic blocks is defined by
\begin{equation*}
\forall\, u\in \mathcal{S'}(\R^d),\quad \Delta_ju=0,\; \text{if}\; j\leq-2;\quad
\Delta_{-1}u=\vartheta(D)u;\quad
\Delta_ju=\varphi(2^{-j}D)u,\; \; \text{if}\;j\geq0,
\end{equation*}
where the pseudo-differential operator is defined by $\sigma(D):u\to\mathcal{F}^{-1}(\sigma \mathcal{F}u)$.

Finally we recall the definition of the nonhomogeneous Besov spaces
$$
B^{s}_{p,r}:=\f\{f\in \mathcal{S}':\;\|f\|_{B^{s}_{p,r}(\R^d)}:=\left\|2^{js}\|\Delta_jf\|_{L_x^p(\R^d)}\right\|_{\ell^r(j\geq-1)}<\infty\g\}.
$$
\section{Proof of Theorem \ref{th1}}\label{sec3}
We divide the proof of Theorem \ref{th1} into three steps.  Now, by the heat flow, we shall mean the unique $L^2$ solution $U=e^{\ep t\De}u^{\ep}_0(x)$ of
\begin{equation}\label{2}
\begin{cases}
\pa_tU-\ep\Delta U=0,\\
\D U=0,\\
U(0,x)=u^{\ep}_0(x).
\end{cases}
\end{equation}

From \eqref{1} with $f^\ep=0$ and \eqref{2}, we deduce that the difference $v:=u^\ep-U$ satisfies the equation with zero initial data
\begin{align}\label{v}
\partial_tv+v\cd\na v-\ep\De v+\na p^\ep =-v\cd\na U-U\cd\na v-U\cd\na U.
\end{align}

{\bf Step 1:\; Estimation of Difference.}
Multiplying Eq. \eqref{v} by $v$ and  integrating in space, then using the fact that $v$ is divergence and the Cauchy-Schwarz inequality we conclude
\bbal
\|v\|^2_{L^2(\R^d)}+\ep\int^t_0\|\na v\|^2_{L^2(\R^d)}\dd \tau &\leq 2\int^t_0(\|\na U\|_{L^\infty(\R^d)}+1)\|v\|^2_{L^2(\R^d)}\dd \tau+\int^t_0\|U\cd\na U\|^2_{L^2(\R^d)}\dd \tau,
\end{align*}
which implies
\bbal
\|v\|^2_{L^2(\R^d)}+\ep\int^t_0\|\na v\|^2_{L^2(\R^d)}\dd \tau \leq  \exp\f(\int^t_02(\|\na U\|_{L^\infty(\R^d)}+1)\dd \tau\g)\int^t_0\|U\|^2_{L^2(\R^d)} \|\na U\|^2_{L^\infty(\R^d)} \dd \tau.
\end{align*}
Hence, noting that $u^\ep=U+v$, we have
\bal\label{lyz}
\ep\int^t_0\|\na u^\ep\|^2_{L^2(\R^d)}\dd \tau&\geq \ep\int^t_0\|\na U\|^2_{L^2(\R^d)}\dd \tau-\ep\int^t_0\|\na v\|^2_{L^2(\R^d)}\dd \tau
\nonumber\\
&\geq \ep\int^t_0\|\na U\|^2_{L^2(\R^d)}\dd \tau\nonumber\\
&\quad-\exp\f(\int^t_02(\|\na U\|_{L^\infty(\R^d)}+1)\dd \tau\g)\int^t_0\|U\|^2_{L^2(\R^d)} \|\na U\|^2_{L^\infty(\R^d)} \dd \tau.
\end{align}

{\bf Step 2:\; Construction of initial data.} We need to introduce smooth, radial cut-off functions to localize the frequency region. Let $\widehat{\phi}\in \mathcal{C}^\infty_0(\mathbb{R})$ be an even, real-valued and non-negative function on $\R$ and satisfy
\begin{numcases}{\widehat{\phi}(\xi)=}
1, &if\; $|\xi|\leq \frac{1}{4}$,\nonumber\\
0, &if\; $|\xi|\geq \frac{1}{2}$.\nonumber
\end{numcases}
\begin{remark} It is not difficult to see that
$$\|\phi\|_{L^{2}(\R)}\thickapprox1\quad\text{and}\quad\|\phi\|_{L^{\infty}(\R)}\thickapprox1.$$
\end{remark}
From now on, we set $\ep_n:=2^{-2n}$ and
$$f_n(x):=\cos \left(\frac{17}{12}2^nx_1\right)\prod_{i=1}^d\phi\f(\ep_nx_i\g).$$
We define the initial data $u^{\ep_n}_0$ by
\begin{align*}
&u^{\ep_n}_0(x):=2^{-dn-n}
\f(
-\pa_2f_n,\;
\pa_1f_n,\;
0,\;
\cdots,\;
0
\g).
\end{align*}
Then we have the following
\begin{lemma}\label{le1} We have
$$\|u^{\ep_n}_0\|_{L^2(\R^d)}\thickapprox \|\phi\|^d_{L^{2}(\R)}\quad\text{and}\quad\|u^{\ep_n}_0\|_{B^1_{\infty,1}(\R^d)}\thickapprox 2^{(1-d)n}\|\phi\|^d_{L^{\infty}(\R)}.$$
\end{lemma}
\begin{proof}
Due to the simple fact $\cos\theta=\frac{\mathrm{1}}{2} (e^{-\mathrm{i} \theta}+e^{\mathrm{i} \theta})$, we deduce easily that
\begin{align*}
\mathcal{F}\left(f_n\right)
=\;&2^{2nd-1} \int_{\R}\left(e^{-\mathrm{i} 2^{2 n} x_1\left(\xi_1+\frac{17}{12} 2^{n}\right)}+e^{-\mathrm{i} 2^{2n} x_1\left(\xi_1-\frac{17}{12} 2^{n}\right)}\right) \phi(x_1) \dd x_1\prod_{i=2}^d\widehat{\phi}\f(2^{2n}\xi_i\g)\\
=\;&2^{2nd-1} \left[\widehat{\phi}\left(2^{2n} \left(\xi_1+\frac{17}{12} 2^{n}\right)\right)+\widehat{\phi}\left(2^{2n} \left(\xi_1-\frac{17}{12} 2^{n}\right)\right)\right]\prod_{i=2}^d\widehat{\phi}\f(2^{2n}\xi_i\g),
\end{align*}
which implies that
\bal\label{con}
\mathrm{supp} \ \mathcal{F}\left(u_0^n\right)=\mathrm{supp} \ \mathcal{F}\left(f_n\right)
&\subset \left\{\xi\in\R^d: \ \frac{33}{24}2^{n}\leq |\xi|\leq \frac{35}{24}2^{n}\right\}.
\end{align}
Recalling that $\varphi\equiv 1$ for $\frac43\leq |\xi|\leq \frac32$, we have
\begin{numcases}{\Delta_j(u^{\ep_n}_0)=}
u^{\ep_n}_0, &if $j=n$,\nonumber\\
0, &if $j\neq n$.\nonumber
\end{numcases}
Thus, we deduce that
\bbal
\|u^{\ep_n}_0\|_{B^\sigma_{p,r}(\R^d)}&=2^{n\sigma}\|u^{\ep_n}_0\|_{L^{p}(\R^d)}
= 2^{n(\sigma-d-1)}\f(\|\pa_1f_n\|_{L^{p}(\R^d)}+\|\pa_2f_n\|_{L^{p}(\R^d)}\g),
\end{align*}
which completes the proof of Lemma \ref{le1}.
\end{proof}
{\bf Step 3:\; Completion of the proof.}
Easy computation show that for $t\in[0,T_{\ep_n})$ (for more details see \cite{GZ})
\begin{align*}
\|u^{\ep_n}(t)\|_{B^1_{\infty,1}(\R^d)}&\les\exp\f(C\int_0^t \|u^{\ep_n}(\tau)\|_{B^1_{\infty,1}(\R^d)}\dd\tau\g) \f(\|u^{\ep_n}_0\|_{B^1_{\infty,1}(\R^d)}+\int^t_0\|\nabla p^{\ep_n}(\tau)\|_{B^1_{\infty,1}(\R^d)}\dd \tau\g)
\\&\les\exp\f(C\int_0^t \|u^{\ep_n}(\tau)\|_{B^1_{\infty,1}(\R^d)}\dd\tau\g)\f(\|u^{\ep_n}_0\|_{B^1_{\infty,1}(\R^d)}+\int^t_0\|u^{\ep_n}(\tau)\|^2_{B^1_{\infty,1}(\R^d)}\dd \tau\g).
\end{align*}
Since $\|u^{\ep_n}_0\|_{B^1_{\infty,1}(\R^d)}\leq C2^{(1-d)n}$, by Lemma \ref{le1} and the classical local-wellposdness, we know that there exists $T_{\ep_n}\geq1$ such that \eqref{1} has a unique solution $u^{\ep_n}$ in $\mathcal{C}([0,1];H^\infty)$.

Noticing that $U_n(t,x)=e^{\ep_n t\De}u^{\ep_n}_0(x)$, from Lemma \ref{le1} we obtain that for some uniform positive constant $C$
\bbal
&\|U_n\|_{L^\infty_T(L^2(\R^d))}\leq \|u^{\ep_n}_0\|_{L^2(\R^d)}\les 1
\end{align*}
and
\bbal
&\|\na U_n\|_{L^\infty_T(L^\infty(\R^d))}\leq \|u^{\ep_n}_0\|_{B^1_{\infty,1}(\R^d)}\les 2^{(1-d)n},
\end{align*}
which implies
\bbal
\exp\f(\int^t_02(\|\na U_n\|_{L^\infty(\R^d)}+1)\dd \tau\g)\int^t_0\|U_n\|^2_{L^2(\R^d)} \|\na U_n\|^2_{L^\infty(\R^d)} \dd \tau
\les 2^{2(1-d)n}.
\end{align*}
Then, coming bake to \eqref{lyz}, we conclude
\bal\label{ad0}
\ep_n\int^1_0\|\na u_n\|^2_{L^2(\R^d)}\dd \tau\geq \ep_n\int^1_0\|\na U_n\|^2_{L^2(\R^d)}\dd \tau-C2^{2(1-d)n}.
\end{align}
Due to the support condition \eqref{con} and the Plancherel's identity, it holds that
\bbal
\ep_n\int^1_0\|\na U_n(\tau)\|^2_{L^2(\R^d)}\dd \tau\approx \int^1_0\f\|e^{\ep_n \tau\De}u^{\ep_n}_0\g\|^2_{L^2(\R^d)}\dd \tau\approx \f\|u^{\ep_n}_0\g\|^2_{L^2(\R^d)}\approx \f\|\phi\g\|^{2d}_{L^2(\R)}.
\end{align*}
Inserting the above into \eqref{ad0} gives the anomalous dissipation statements \eqref{ad}-\eqref{ad1}. Thus we complete the proof of Theorem \ref{th1}.
\section*{Acknowledgements}
J. Li is supported by the National Natural Science Foundation of China (11801090 and 12161004) and Jiangxi Provincial Natural Science Foundation (20212BAB211004 and 20224BAB201008). Y. Yu is supported by the National Natural Science Foundation of China (12101011). W. Zhu is supported by the National Natural Science Foundation of China (12201118) and Guangdong Basic and Applied Basic Research Foundation (2021A1515111018).

\section*{Declarations}
\noindent\textbf{Data Availability} No data was used for the research described in the article.

\vspace*{1em}
\noindent\textbf{Conflict of interest}
The authors declare that they have no conflict of interest.


\addcontentsline{toc}{section}{References}
\begin{thebibliography}{99}
\linespread{0}\addtolength{\itemsep}{-1.0ex}


\bibitem{BCD} H. Bahouri, J.-Y. Chemin, R. Danchin, Fourier Analysis and Nonlinear Partial Differential Equations, Grundlehren der Mathematischen Wissenschaften, 343, Springer, 2011.

\bibitem{EL} E. Bru\`{e}, C.-D. Lellis, Anomalous dissipation for the forced 3D Navier–Stokes
equations, Commun. Math. Phys. 400, 1507-1533 (2023).
\bibitem{BV} T. Buckmaster, V. Vicol, Convex integration and phenomenologies in turbulence, EMS Surv. Math.
Sci. 6(1), 173–263 (2019).

\bibitem{DE} T.-D. Drivas, T.-M. Elgindi, G. Iyer, I.-J. Jeong, Anomalous dissipation in passive scalar transport, Arch. Ration. Mech. Anal. 243(3), 1151-1180 (2022).
\bibitem{E} G.-L. Eyink, Review of the Onsager ``Ideal Turbulence" Theory, arXiv:1803.02223.
\bibitem{u} Uriel Frisch, Turbulence, Cambridge University Press, Cambridge, 1995. The legacy of A. N.
Kolmogorov.
\bibitem{GZ} Z. Guo, J. Li, Z. Yin, Local well-posedness of the incompressible Euler equations in $B_{\infty,1}^1$ and the inviscid limit of the Navier--Stokes equations, J. Funct. Anal., 276 (2019), 2821-2830.
\bibitem{JY1} I.-J. Jeong, T. Yoneda, Quasi-streamwise vortices and enhanced dissipation for incompressible
3D Navier-Stokes equations, Proc. Am. Math. Soc. 150(3), 1279-1286 (2022).
\bibitem{JY2} I.-J.Jeong, T. Yoneda, Vortex stretching and enhanced dissipation for the incompressible 3D
Navier-Stokes equations, Math. Ann. 380(3–4), 2041-2072 (2021).
\bibitem{KI} Y. Kaneda, T. Ishihara, M. Yokokawa, K. Itakura, A. Uno, Energy dissipation rate and energy spectrum in high resolution direct numerical simulations of turbulence in a periodic box, Phys Fluids 15(2), L21–L24 (2003).
\bibitem{K1} A. Kolmogorov, Dissipation of energy in the locally isotropic turbulence, Dokl. Akad. Nauk SSSR
32, 16–18 (1941)
\bibitem{K2} A. Kolmogorov, Local structure of turbulence in an incompressible fluid at very high reynolds number, Dokl. Acad. Nauk SSSR 30(4), 299-303 (1941).
\bibitem{K3} A. Kolmogorov, On degeneration of isotropic turbulence in an incompressible viscous liquid, Dokl. Akad. Nauk SSSR 31, 538–540 (1941).

\bibitem{PKW} B.R. Pearson, P.-{\AA}. Krogstad, W. van de Water, Measurements of the turbulent energy dissipation rate, Physics of Fluids 14, 1288-1290 (2002).
\bibitem{KR1} K.R. Sreenivasan, On the scaling of the turbulence energy dissipation rate, Physics of Fluids 27, 1048-1051 (1984).
\bibitem{KR2} K.R. Sreenivasan, An update on the energy dissipation rate in isotropic turbulence, Physics of Fluids 10, 528-529 (1998).

\bibitem{J15} J. Vassilicos, Dissipation in turbulent flows, Annual Review of Fluid Mechanics 47, 95-114 (2015).





\end{thebibliography}
\end{document}